\documentclass{birkjour}
 \newtheorem{thm}{Theorem}[section]

 \theoremstyle{definition}
 
 \theoremstyle{remark}

 \numberwithin{equation}{section}
\usepackage{hyperref}

\begin{document}

%
%
%
%
%
%
%
%
%

\title[DUAL HORADAM OCTONIONS]
 {DUAL HORADAM OCTONIONS}

\author[Serp\.{i}l HALICI]{Serp\.{i}l HALICI}

\address{%
Pamukkale University,\\
Faculty of Arts and Sciences,\\
Department of Mathematics,\\
Denizli/TURKEY}

\email{shalici@pau.edu.tr}

\author[Adnan KARATA\c{S}]{Adnan KARATA\c{S}}

\address{%
Pamukkale University,\\
Faculty of Arts and Sciences,\\
Department of Mathematics,\\
Denizli/TURKEY}

\email{adnank@pau.edu.tr}



\date{Feb, 2017}

\begin{abstract}
In this study, we investigate Horadam sequence as generalization of linear recurrence equations of order two. By the aid of this sequence we obtain a new generalization for sequences of dual quaternions and dual octonions. Moreover, we derive some important identities such as Binet formula, generating function, Cassini identity, sum formula and norm formula by their Binet forms. We generalize all studied linear second order recurrence relations over dual octonions and quaternions in a single formula.
\end{abstract}

\maketitle
\vspace{1.5 cm}
\section{Introduction}

W.R. Hamilton introduced quaternions as an extension of complex numbers. According to Hamilton any quaternion $q$ is a hyper-complex number represented by an equation
$$ q= a_{0}e_{0}+ a_{1}e_{1}+ a_{2}e_{2}+ a_{3}e_{3}$$ where $ a_{0}, a_{1}, a_{2},  a_{3}\in  \mathbb{R}$  and $ \{ e_{0}, e_{1}, e_{2},e_{3} $ which forms a standard orthonormal basis in $\mathbb{R}^{4} $. Horadam, in \cite{Hor2}, defined Fibonacci and Lucas quaternions as follows.\\
$$ QF_{n} = F_{n}e_{0}+ F_{n+1}e_{1} + F_{n+2}e_{2}+ F_{n+3}e_{3} $$
$$ QL_{n} = L_{n}e_{0}+ L_{n+1}e_{1} + L_{n+2}e_{2}+ L_{n+3}e_{3} $$  \\ where $ F_{n}, L_{n}$ denote $n$-$th$ Fibonacci and Lucas numbers, respectively. In \cite{Hor4}, the author introduced the sequence $\{w_n (a, b; p, q) \}$ and gave its basic properties. Also, in \cite{Hor6} he examined some of the fundamental properties of $w_n$. Swamy \cite{Swamy} defined the generalization of Fibonacci quaternions. Halici, in \cite{Hal1, Hal2} , investigated Fibonacci quaternions and complex Fibonacci quaternions. The author derived some identities including Binet formula, generating functions and binomial sums involving these quaternions. Polatli and Kesim \cite{Polatli}, studied the quaternions with the generalized Fibonacci and Lucas number components. Since a dual quaternion is usually described as a quaternion with dual numbers as coefficient, dual Fibonacci quaternions and octonions can be defined in a similar way. That is a dual quaternion $\hat{P} $ can be written as
$$\hat{P}=p + q \epsilon;\,\ \,\ \,\ \epsilon^2=0,\,\ \,\ \epsilon \neq 0 \mbox{  and  } p, q \in \mathbb{H}.$$
Addition operation on the dual quaternions is component-wise and multiplication operation is as follows
$$\hat{P} \hat{Q} = (p,q)(r,s)=(p r, p s+q r)=p r+ (p s +q r) \epsilon.$$
Like dual quaternions, dual octonions are also useful tool for geometry \cite{H.K.}and electromagnetism \cite{Chanyal}. Reader who wants to make further reading about dual quaternions in geometry and programming can refer to \cite{Wang, Dani}.\\
Now let us give some fundamental properties of dual octonions. Any dual octonion can be defined as $\hat{K}=(k, l)=k+l \epsilon $ where $ k, l \in \mathbb{O} \mbox{  with  } \epsilon^2 =0.$
Let $\mathbb{D}$ be the set of dual octonions;
$$\mathbb{D}=\{\hat{K}|\hat{K}=(k,l)=k+l \epsilon; \,\ \,\ k,l \in \mathbb{O} \mbox{  with  } \epsilon^2 =0 \}.$$
Addition operation on $\mathbb{D}$ is
$$\hat{K}+\hat{L}=(k,l)+(m,n)=(k+m,l+n)=(k+m)+(l+n)\epsilon .$$
To indicate multiplication of two octonions we will use the notation $\circ$. As it is well-known there are different multiplication tables for octonion multiplication (see, \cite{Baez}). We use the multiplication rules according to following sets and to determine indices $\{ i, j, k \} $ we can choose from followings $ \{ 1, 2, 3 \},$ $\{ 1, 4, 5 \},$ $ \{ 1, 7, 6 \},$ $ \{ 2, 4, 6 \},$ $\{ 2, 5, 7 \},$ $\{ 3, 4, 7 \}$ and $\{ 3, 6, 5 \}$.
$$e_i \circ e_j= e_k.$$
We note that if one chooses different multiplication table then should make all calculations accordingly and most likely get different results from our results. Analogously, multiplication of two dual octonions is as follows
$$\widehat{K} \widehat{L}=(k \circ m, k \circ n + l \circ m ).$$
Conjugate of dual octonion $\widehat{K}$ can be demonstrated as $\overline{\widehat{K}}$. And it can be defined as follows
$$\overline{\widehat{K}}=(\overline{k},\overline{l})=\overline{k} +\overline{l} \epsilon .$$
Norm is defined by the help of conjugate as follows
$$N(\widehat{K})= (k,l)(\overline{k},\overline{l})=(\overline{k},\overline{l})(k,l)=(k \circ \overline{k},k\circ \overline{l}+ l \circ \overline{k} ).$$
It should be noted that, definition of dual unit $\epsilon$ causes norm to be degenerate which means that the norm can be equal to zero for some non zero dual octonions. To explain this situation let us calculate the norm for $\widehat{K}=1\epsilon$;
$$N(\widehat{K})= 1 \epsilon \overline{1} \epsilon = 1 \epsilon^2=0.$$
Thus, this example shows the norm is degenerate. For more details, we refer the interested reader to \cite{Baez} for octonions and \cite{Conway} for norms.\\\\
In this study, we focus on the second order recursive relations over dual octonions. There are some papers about this subject in literature. In \cite{S.H}, the author studied dual Fibonacci octonions. In  \cite{zafer}, the author investigated dual Fibonacci and dual Lucas octonions. Our aim is to generalize all of second order recursive relations over dual octonions. To achieve this goal firstly we will give some properties of Horadam sequence in section $2$, and then we will introduce dual Horadam octonions in section $3$ and finally we will conclude this paper in section $4$.
%
%
%
%
\section{\textbf{HORADAM SEQUENCE}}

Horadam sequence is a particular type of linear recurrence sequences. This sequence generalizes some of the well-known sequences such as Fibonacci, Lucas, Pell, Pell-Lucas sequences, etc.. There are many studies in the literature concerned about the Horadam sequence (see for instance \cite{Hauk, Hor6, Hor4}).\\\\
Now we present some formulas that are used subsequently. Let us start with the definition of Horadam numbers. In \cite{Hor6}, Horadam defined the following sequence.
\begin{equation} \{ w_n(a,b;p,q) \}; w_n = p w_{n-1} + q w_{n-2}; w_0=a, w_1=b, \,\ \,\  (n \geq 2). \end{equation}
In the equation $(2.1)$ if we take $a=0, b=p=q=1$, then we get Fibonacci numbers. Likewise Lucas, Pell and other sequences can be obtained.\\\\
Some needed properties of Horadam numbers will be given below.
One of very important properties about recurrence relations is Binet formula and it can be calculated using its characteristic equation. The roots of characteristic equation are
\begin{equation} \alpha = \frac{p+\sqrt{p^2+4q}}{2}, \,\ \,\ \beta= \frac{p-\sqrt{p^2+4q}}{2}. \end{equation}
Using these roots and the recurrence relation of Horadam numbers, Binet formula can be given as follows
\begin{equation} w_n=\frac{A\alpha^n-B\beta^n}{\alpha - \beta}, \,\ \,\ \,\ A=b-a\beta \mbox{ and } B=b-a\alpha.\end{equation}
The generating function for Horadam numbers is
\begin{equation} g(t)=\frac{w_0+(w_1-pw_0)t}{1-pt-qt^2}. \end{equation}
The Cassini identity for Horadam numbers is
\begin{equation} w_{n+1}w_{n-1}-w^2_n=q^{n-1}(pw_0 w_1- w_1^2-w_0^2q). \end{equation}
Moreover, a summation formula for Horadam numbers is
\begin{equation} \sum_{i=0}^n{w_i}=\frac{w_1-w_0 (p-1)+qw_n - w_{n+1}}{1-p-q}.\end{equation}
These equalities will be our base point and then we will investigate dual Horadam octonions accordingly. There are some studies about second order recurrence relations over $\mathbb{R, C, H}$ and $\mathbb{O}$ in literature. One of the earliest studies on this topic belong to Horadam where it was given a generalization of second order recurrence relations on $\mathbb{R, C}$. Swamy defined the generalized Fibonacci quaternions and then Halici defined quaternions with coefficients from Horadam sequence \cite{Halici}. Also, in \cite{Kar.Hal.}, Karata\c{s} and Halici defined Horadam octonions. In addition they gave the Binet formula and some identities. Especially, in \cite{Aslan, zafer, S.H}, authors investigated the second order recurrence relations on dual structures.\\\\
%
%
%
\section{\textbf{DUAL HORADAM OCTONIONS}}
In \cite{Kar.Hal.}, we defined and investigated Horadam octonions. Moreover, we derived Binet formula, generating function, Cassini identities, sum and norm formulas for Horadam octonions.\\\\
In this section, we define dual Horadam octonions and analogously give mentioned identities and equalities. Also, we show identities and equalities for dual Horadam octonions that can be reduced to identities and equalities in \cite{S.H} and \cite{zafer} via proper initial values.\\\\
Dual Horadam octonions are dual octonions with coefficients from Horadam sequence. Let $\widehat{\mathbb{O}G_{n}}$ be $n$-$th$ dual Horadam octonion. We define this type of octonion as
$$\widehat{\mathbb{O}G_{n}}=(\mathbb{O}G_{n}, \mathbb{O}G_{n+1})=\mathbb{O}G_{n}+\mathbb{O}G_{n+1} \epsilon$$
where $$\mathbb{O}G_{n}=w_n e_0+ w_{n+1} e_1 + \dots + w_{n+7} e_7. $$
Now, let us examine some properties of dual Horadam octonions. The following theorem gives us the Binet formula for dual Horadam octonions.
\begin{thm}
For $n\geq 1$, we get  \\
\begin{equation}\widehat{\mathbb{O}G_{n}}=\frac{A\underline{\alpha}\alpha^n(1+\alpha \epsilon)-B\underline{\beta}\beta^n(1+\beta \epsilon)}{\alpha - \beta}, \end{equation} \\
where $\underline{\alpha}=1e_0 + \alpha e_1+ \alpha^2 e_2+ \dots + \alpha^7 e_7$ and $\underline{\beta}=1e_0 + \beta e_1+ \beta^2 e_2+ \dots + \beta^7 e_7.$
\end{thm}
\begin{proof}
For these octonions Binet formula can be calculated similar to Binet formula for Horadam numbers. By using characteristic equation
\begin{equation}t^2-pt-q=0.\end{equation}
The roots of characteristic equation are
\begin{equation} \alpha = \frac{p+\sqrt{p^2+4q}}{2}, \,\ \,\ \beta= \frac{p-\sqrt{p^2+4q}}{2}. \end{equation}
Using these roots and the recurrence relation we write Binet formula as follows.
\begin{equation}\mathbb{O}G_{n}=\frac{A\underline{\alpha}\alpha^n-B\underline{\beta}\beta^n}{\alpha - \beta}\end{equation}
where
\begin{equation}\underline{\alpha}=1e_0 + \alpha e_1+ \alpha^2 e_2+ \dots + \alpha^7 e_7\mbox{ and }\underline{\beta}=1e_0 + \beta e_1+ \beta^2 e_2+ \dots + \beta^7 e_7.\end{equation}
When using the definition of dual Horadam octonions we have
$$\widehat{\mathbb{O}G_{n}}=\frac{A\underline{\alpha}\alpha^n(1+\alpha \epsilon)-B\underline{\beta}\beta^n(1+\beta \epsilon)}{\alpha - \beta}$$
which completes the proof.
\end{proof}
In \cite{zafer}, Binet formula for dual Fibonacci octonions is given and as follows
$$ \frac{\underline{\alpha}\alpha^n(1+\alpha \epsilon)-\underline{\beta}\beta^n(1+\beta \epsilon)}{\alpha - \beta}.$$
By the aid of necessary initial values it can be seen that equation $(3.1)$ equals to Binet formula in \cite{zafer}. This fact shows our Binet formula is a generalization of Binet formulas in \cite{S.H, zafer}.\\\\
In the next theorem we consider the generating function.
\begin{thm}
The generating function for dual Horadam octonions is \\
\begin{equation}\frac{\widehat{\mathbb{O}G_{0}}+(\widehat{\mathbb{O}G_{1}}-p\widehat{\mathbb{O}G_{0}})t}{1-pt-qt^2}.\end{equation}
\end{thm}
\begin{proof}
To prove this claim, firstly, we need to write generating function for dual Horadam octonions;
\begin{equation}g(t)=\widehat{\mathbb{O}G_{0}}t^0+ \widehat{\mathbb{O}G_{1}}t + \dots+ \widehat{\mathbb{O}G_{n}}t^n+\dots \end{equation} \\
Secondly, we need to calculate $ptg(t)$ and $qt^2g(t)$ as the following equations;\\
\begin{equation}ptg(t)=\sum^\infty_{n=0}{p\widehat{\mathbb{O}G_{n}}t^{n+1}} \mbox{ and } qt^2g(t)=\sum^\infty_{n=0}{q\widehat{\mathbb{O}G_{n}}t^{n+2}}. \end{equation}\\
Finally, if we made necessary calculations, then we have \\
\begin{equation}g(t)=\frac{\mathbb{O}G_{0}+\mathbb{O}G_{1} \epsilon+(\mathbb{O}G_{1} + \mathbb{O}G_{2} \epsilon - p ( \mathbb{O}G_{0} + \mathbb{O}G_{1} \epsilon))t}{1-pt-qt^2}\end{equation}which is the generating function for dual Horadam octonions.
\end{proof}
In \cite{S.H}, the author gave generating function of dual Fibonacci octonions as follows.
$$ \frac{\widehat{\mathbb{O}_{0}} + (\widehat{\mathbb{O}_{1}} -  \widehat{\mathbb{O}_{0}})t}{1-t-t^2}.$$
If we choose the values as $a=0, b=p=q=1$ in the equation $(3.9)$ we get the generating function of dual Fibonacci octonions.
\begin{thm}
For dual Horadam octonions the Cassini formula is 
\begin{equation} \widehat{\mathbb{O}G_{n-1}} \widehat{\mathbb{O}G_{n+1}}-\widehat{\mathbb{O}G^2_{n}} =(c_1, c_2)\end{equation}
where
$$c_1 = \frac{AB(\alpha \beta)^{n-1}(\beta \underline{\alpha}\underline{\beta}-\alpha \underline{\beta}\underline{\alpha})}{\alpha-\beta}$$
and
$$c_2 = \frac{A^2 {\underline{\alpha}}^2 \alpha^{2n} (\alpha -1)^2 + B^2 \underline{\beta}^2 \beta^{2n} (\beta - 1 )^2}{(\alpha-\beta)^2} $$
$$-AB(\alpha \beta)^{n-1}((\beta^2 + \alpha \beta^3 - \alpha \beta^2 - \alpha^2 \beta)\underline{\alpha}\underline{\beta} +(\alpha^2 + \alpha^3 \beta - \alpha^2 \beta - \alpha \beta^2)\underline{\beta}\underline{\alpha}).$$
\end{thm}
\begin{proof}
Let us use Binet formula to prove equation $(3.10)$ \\
$$c_1= \frac{A\underline{\alpha}\alpha^{n-1}-B\underline{\beta}\beta^{n-1}}{\alpha - \beta} \frac{A\underline{\alpha}\alpha^{n+1}-B\underline{\beta}\beta^{n+1}}{\alpha - \beta}- \bigg( \frac{A\underline{\alpha}\alpha^{n}-B\underline{\beta}\beta^{n}}{\alpha - \beta} \bigg)^2.$$ \\
If necessary calculations are made, we obtain \\
$$c_1=\frac{AB(\alpha \beta)^{n-1}(\beta \underline{\beta}\underline{\alpha}-\alpha \underline{\alpha}\underline{\beta})}{\alpha-\beta}$$ \\
which is desired. In a similar way, the dual part $c_2$ can be easily obtained.
\end{proof}
Cassini identities for dual Fibonacci octonions and dual Lucas octonions are studied in \cite{S.H, zafer}. To obtain Cassini identity for dual Fibonacci octonions one can choose necessary initial values. Thus, one can get the following equation which is equal to Cassini identity in \cite{S.H}. Notice that we show $n$-$th$ dual Fibonacci octonions as $\widehat{\mathbb{O}_{n}}$.
$$\widehat{\mathbb{O}_{n-1}} \,\ \widehat{\mathbb{O}_{n+1}}-\widehat{\mathbb{O}^2_{n}} =(-1)^n (\widehat{\mathbb{O}_{1}^2}-\widehat{\mathbb{O}_{0}^2}-\widehat{\mathbb{O}_{1}}\widehat{\mathbb{O}_{0}}).$$
\begin{thm}
The sum formula for dual Horadam octonions is as follows.\\
\begin{equation} \sum_{i=0}^n{\widehat{\mathbb{O}G_{i}}}= \sum_{i=0}^n{\mathbb{O}G_i}+\sum_{i=1}^{n+1}{\mathbb{O}G_i\epsilon} =(d_1, d_2)  \end{equation}
where
$$d_1 = \sum_{i=0}^n{\mathbb{O}G_i}= \frac{1}{\alpha-\beta}\big( \frac{B\underline{\beta}\beta^{n+1}}{1-\beta}-\frac{A\underline{\alpha}\alpha^{n+1}}{1-\alpha} \big) + K$$
and
$$d_2 = \sum_{i=1}^{n+1}{\mathbb{O}G_i}=\frac{1}{\alpha-\beta}\big( \frac{B\underline{\beta}\beta^{n+2}}{1-\beta}-\frac{A\underline{\alpha}\alpha^{n+2}}{1-\alpha} \big) - \mathbb{O}G_0+ K  $$
where $K$ is \\
$$K=\frac{A\underline{\alpha}(1-\beta)-B\underline{\beta}(1-\alpha)}{(\alpha-\beta)(1-\alpha)(1-\beta)}.$$
\end{thm}
\begin{proof}
For the summation formula of dual Horadam octonions we calculate the values $d_1$ and $d_2$, separately. To calculate $d_1$ we can use Binet formula as follows.
$$d_1 = \sum_{i=0}^n{\mathbb{O}G_i}=\frac{A\underline{\alpha}\alpha^{n}-B\underline{\beta}\beta^{n}}{\alpha-\beta}=\frac{A\underline{\alpha}}{\alpha-\beta} \sum_{i=0}^n \alpha^{n}-\frac{B\underline{\beta}}{\alpha-\beta} \sum_{i=0}^n \beta^{n} .$$
Using properties of geometric series we get the following equation.
$$ d_1 = \sum_{i=0}^n{\mathbb{O}G_i}=\frac{A\underline{\alpha}}{\alpha-\beta}\sum_{i=0}^n \frac{1-\alpha^{n+1}}{1-\alpha}-\frac{B\underline{\beta}}{\alpha-\beta} \sum_{i=0}^n \frac{1-\beta^{n+1}}{1-\beta}. $$
After making necessary calculations we will get the explicit form of $d_1$ as
$$\sum_{i=0}^n{\mathbb{O}G_i}=\frac{1}{\alpha-\beta}\big( \frac{B\underline{\beta}\beta^{n+1}}{1-\beta}-\frac{A\underline{\alpha}\alpha^{n+1}}{1-\alpha} \big) +\frac{A\underline{\alpha}(1-\beta)-B\underline{\beta}(1-\alpha)}{(\alpha-\beta)(1-\alpha)(1-\beta)}.$$
From the definition of dual Horadam octonions the value of $d_2$ can be easily calculated.
\end{proof}
In \cite{S.H}, the author gave the sum formula for dual Fibonacci octonions as 
$$ \sum_{i=1}^n{\widehat{\mathbb{O}_{i}}}=\widehat{\mathbb{O}_{2}} F_{n+1} + \widehat{\mathbb{O}_{1}} F_{n-1} - \widehat{\mathbb{O}_{2}}$$
where $F_n$ is the $n$-$th$ Fibonacci number. If we choose $a=0, b=p=q=1$ then sum formula of dual Horadam octonions will be reduced to the sum formula of dual Fibonacci octonions.
\begin{thm}
The norm of $nth$ dual Horadam octonion is
\begin{equation} Nr(\widehat{\mathbb{O}G_{n}})=(\mathbb{O}G_{n} \circ \overline{\mathbb{O}G_{n}},\mathbb{O}G_{n} \circ \overline {\mathbb{O}G_{n+1}} + \mathbb{O}G_{n+1} \circ \overline{\mathbb{O}G_{n}})=(e_1, e_2) \end{equation}
where $e_1, e_2$ and $L$ is
$$e_1 = \frac{A^2\alpha^{2n}(1+\alpha^2+\alpha^4+\dots+\alpha^{14})+B^2\beta^{2n}(1+\beta^2+\beta^4+\dots+\beta^{14})}{(\alpha-\beta)^2}-L,$$
$$e_2 = 2(\sum_{i=0}^7{w_{n+i} w_{n+1+i}})$$
and
$$L=\frac{2AB(-q)^n(a+(-q)+\dots+(-q)^7)}{(\alpha- \beta)^2}.$$

\end{thm}
\begin{proof}
It follows from the definition of norm,
$$Nr(\widehat{\mathbb{O}G_{n}})=(\mathbb{O}G_{n} \circ \overline{\mathbb{O}G_{n}},\mathbb{O}G_{n} \circ \overline {\mathbb{O}G_{n+1}} + \mathbb{O}G_{n+1} \circ \overline{\mathbb{O}G_{n}}).$$
To get desired results we firstly calculate $e_1$.
$$e_1=\frac{A^2\alpha^{2n}(1+\alpha^2+\alpha^4+\dots+\alpha^{14})}{(\alpha-\beta)^2} + $$
$$\frac{B^2\beta^{2n}(1+\beta^2+\beta^4+\dots+\beta^{14})}{(\alpha-\beta)^2} - \frac{2AB(-q)^n(a+(-q)+\dots+(-q)^7)}{(\alpha- \beta)^2}.$$
Dual part of equation can be easily calculated because of the property $\epsilon^2=0$. After direct calculations we get 
$$e_2 = \mathbb{O}G_{n} \circ \overline {\mathbb{O}G_{n+1}} + \mathbb{O}G_{n+1} \circ \overline{\mathbb{O}G_{n}}=2(\sum_{i=0}^7{w_{n+i} w_{n+1+i}}). $$
\end{proof}
Norm values of dual Fibonacci and dual Lucas octonions are studied in \cite{S.H, zafer}. Utilizing the equation $(3.12)$ with appropriate initial values one can get the following equation which is identical to the norm value of dual Fibonacci octonions in \cite{S.H}.
$$Nr(\widehat{\mathbb{O}_{n}})=21(F_{2n+7}+2 F_{2n+8} \epsilon).$$
\maketitle
\section{\textbf{Conclusion}}
In this study, we generalize second order recurrence relations on dual octonions which means that this generalization valid over dual octonions, real octonions, dual quaternions, and so on. Moreover, this generalization includes all second order recurrence relations such as Fibonacci, Lucas, Pell, etc. By the aid of this generalization, we give fundamental properties such as Binet formula, generating function, Cassini identity, sum formula and norm formula for mentioned relations.
\bibliographystyle{plain}
\bibliography{mybibfile}

\end{document}